\documentclass[12pt]{article}
\usepackage[utf8]{inputenc}
\usepackage[left=1in, right=1in, top=1in,bottom = 2in]{geometry}
\usepackage[shortlabels]{enumitem}
\usepackage{amsmath}
\usepackage{amsthm}
\usepackage{amssymb}
\usepackage{amsfonts}
\usepackage{array}
\usepackage{color}
\usepackage{comment}
\usepackage{dsfont}
\usepackage{easytable}
\usepackage{enumitem}
\usepackage{environ}
\usepackage{fancyhdr}
\usepackage{framed}
\usepackage{graphicx}
\usepackage{hyperref}
\usepackage{listings}
\usepackage{multirow}
\usepackage{tikz-cd}
\usepackage{wasysym}
\usepackage{xcolor}

\newtheorem{theorem}{Theorem}

\newtheorem{lemma}[theorem]{Lemma}

\newtheorem{corollary}[theorem]{Corollary}

\usepackage{arydshln}
\makeatletter
  \renewcommand*\env@matrix[1][*\c@MaxMatrixCols c]{%
    \hskip -\arraycolsep
    \let\@ifnextchar\new@ifnextchar
  \array{#1}}
\makeatother

\newcommand{\Z}{\mathbb{Z}}

\newcommand{\F}{\mathbb{F}}

\begin{document}
\author{
  J. G. Coelho\\
  \and F. E. Brochero Mart\'{\i}nez\\
}
\date{August 2024}
\title{Low-weight codewords in cyclic codes}
\maketitle
\begin{abstract}
We introduce a formula for determining the number of codewords of weight 2 in cyclic codes and provide results related to the count of codewords with weight 3. Additionally, we establish a recursive relationship for binary cyclic codes that connects their weight distribution to the number of solutions of associated systems of polynomial equations. This relationship allows for the computation of weight distributions from known solutions of systems of diagonal equations and vice versa, offering a new insight into the structure and properties of binary cyclic codes.

\end{abstract}

\section{Introduction}

A typical feature of codes is to introduce redundancy into a message that will be transferred through a noisy channel, which can corrupt individual bits of information. Typically, the messages are divided into codewords of fixed length and multiplied by the generator matrix to be transformed into vectors of greater length. The introduced redundancy allows some algorithm specific to the type of code to recover the original message even if a few bits of it are corrupted. The minimum distance is a very important parameter, as it is the minimum degree of ``separation" between codewords in the code, and the greater it is, the more errors can be corrected. 

The weight of a codeword in an error-correcting code is the number of non-zero coordinates it has. For linear codes the smallest weight is $0$ for the trivial codeword. The minimum distance
results in the smallest weight different from zero that a codeword has. For a given code of length $n$, we denote by $A_w$ the number of codewords with weight $w$, and call the sequence $A_0, A_1, \dots, A_n$ the weight distribution of the code. Notably, determining the weight distribution of a code also determines its minimum distance.

We will only discuss linear codes, that is, codes where the codewords form a vector subspace of a vector space with a dimension equal to the length of the code. More specifically, we will focus on cyclic codes, with some results being exclusive to binary cyclic codes. Cyclic codes have properties that are useful for encoding, but the way they are defined in terms of polynomials and ideals makes it difficult to determine their weight distribution. There have been many approaches to determine the weight distribution of various families of cyclic codes. Many of these approaches are listed in the surveys \cite{survey, survey 2}. Most of the research done has been directed at determining the distribution of codes with few possible weights. 

A few works have taken another route: rather than finding the complete weight distribution, they consider codes that may have very complicated weight distributions with many possible weights and instead determine the first few values in the weight distribution. Charpin, Tiet\"av\"ainen and Zinoviev have studied codewords with weight $2$ and $3$ in binary cyclic codes \cite{charpin}. Moisio and Ranto studied low-weight codewords in a particular family of binary cyclic codes \cite{kloosterman moisio}. We remark that determining the number of low-weight codewords can be enough to determine the minimum distance of a code if it is low.

In this paper, we will discuss how to compute the number of low-weight codewords in cyclic codes. In the first part we give a few results relating to counting the number of codewords with weights $2$ and $3$ in cyclic codes. Some of these results are more general in nature and do not require the code to be binary, while others do.
In the section with the main results, we present a recursive relation between the weight distribution of binary cyclic codes and the number of solutions of suitable systems of diagonal equations. The last section presents some applications of our main result.

\section{Preliminaries}

Let $\F_p$ be a finite field with prime cardinality $p$, and let $\F_q$ be a finite field where $q = p^m$ for some integer $m$. Let $\gamma$ be a primitive element of $\F_q$. For any element $\gamma^j \in \F_q^*$, we denote its minimal polynomial over $\F_p$ by $g_j(x)$. For indices $t_1, \dots, t_s$ not in the same $p$-cyclotomic coset, we denote by $C_{t_1, \dots, t_s}$ the cyclic code over $\F_p$ with length $n = q - 1$ and generator polynomial $g_{t_1}(x) g_{t_2}(x) \cdots g_{t_s}(x)$. The following result gives us a straightforward way to check if a codeword is in a cyclic code of this form:
\begin{lemma}\label{criterio palavra em codigo ciclico}
    Let $C_{t_1, \dots, t_s}$ be a cyclic code of length $n$ and $\bar{c}$ be a codeword with polynomial form $c(x)$. Let us define the matrix
    \begin{equation}\label{definicao H}
    H := \begin{bmatrix} 1 & \gamma^{t_1} & \gamma^{2 t_1} & \cdots & \gamma^{(n-1) t_1}\\
\vdots & \vdots & \vdots & \ddots & \vdots \\
1 & \gamma^{t_s} & \gamma^{2 t_s} & \cdots & \gamma^{(n-1) t_s}\end{bmatrix}.
\end{equation}
The following are equivalent:
    \begin{enumerate}[(i)]
        \item $\bar{c}$ is in $C_{t_1, \dots, t_s}$.

        \item $c(\gamma^{t_j}) = 0$ for $1 \le j \le s$.

        \item $H \bar{c}^T = 0$.
    \end{enumerate}
\end{lemma}
\begin{proof}
    Since the code is generated by $g(x) = g_{t_1}(x) \cdots g_{t_s}(x)$, a codeword is in the code if and only if it is divisible by $g(x)$ in polynomial form. However, a polynomial is divisible by $g(x)$ if and only if it has $\gamma^{t_1}, \dots, \gamma^{t_s}$ as roots, hence (i) and (ii) are equivalent.
    
    The condition $c(\gamma^{t_j}) = 0$ can be rewritten as 
    $$\begin{bmatrix}
        1 & \gamma^{t_j} & \gamma^{2 t_j} & \dots & \gamma^{(n - 1) t_j}
    \end{bmatrix}\bar{c}^T = 0.$$
    Hence, a codeword satisfies $c(\gamma^{t_j}) = 0$ for every $1 \le j \le s$ if and only if $H \bar{c}^T = 0$, proving that (ii) and (iii) are equivalent.

\end{proof}

We notice that the matrix $H$ in (\ref{definicao H}) is not a parity check matrix for the code, since the entries are in $\F_q$ instead of $\F_p$. However, it can be converted into a parity check matrix by fixing a base for $\F_q$ over $\F_p$, turning each row into $m$ rows where each coordinate in the $i$-th row is the coefficient of the corresponding element of the original row with respect to the $i$-th basis element, and then excluding rows until the matrix is linearly independent. Hence, we will also refer to such a matrix as the parity check matrix.

\hfill\newline

\section{Weight 2 and 3 codewords}
We will first study low-weight codewords in cyclic codes. Over any finite field, the only cyclic code that has weight $1$ codewords is the code that includes every codeword, and thus there is no need to further investigate this case. Therefore, the question that needs to be answered first is when a cyclic code has codewords with weight $2$. Hence, the first result we will prove is a theorem to determine the number of codewords with weight $2$ in terms of the parameters of a cyclic code. The following theorem provides a necessary and sufficient condition for the minimum distance to be $2$ and determines $A_2$:

\begin{theorem}\label{peso 2} 
Let us consider $q = p^m$ to be a prime power. The code $C_{t_1, \dots, t_s}$ of length $n = q - 1$ has a minimum distance of 2 if and only if the greatest common divisor
\begin{align*}
    D(t_1,\dots, t_s) := \gcd(q - 1, t_1 (p - 1), t_2 - t_1, t_3 - t_1, \dots, t_s - t_1)
\end{align*}
is greater than 1. In that case, the number of codewords with weight 2 is given by
$$A_2 = \frac{(p - 1)(q - 1)(D(t_1, \dots, t_s) - 1)}{2}.$$

If $p = 2$, the expression for $D(t_1, \dots, t_s)$ simplifies to $\gcd(q - 1, t_1, \dots, t_s)$.
\end{theorem}
\begin{proof}
    A codeword of weight 2 is a vector of length $n = q - 1$ over $\F_p$ such that only two coordinates are non-zero. We can shift one of those non-zero values to the first index and multiply by the inverse of this value to make the first coordinate a $1$. Thus, the code has a codeword with weight $2$ if and only if it has a codeword of the form 
    $$c(i, \alpha) = 1 + \alpha x^i,$$
    where $ 1 \le i \le q - 2$ and $\alpha \in \F_p^*$. Multiplying the vector form of this codeword by the matrix $H$ in (\ref{definicao H}), we get that a codeword $c(i, \alpha)$ is in the code if and only if $i$ and $\alpha$ satisfy the system
    $$\begin{cases}
        1 + \alpha \gamma^{i t_1} = 0,\\
        \vdots \\
        1 + \alpha \gamma^{i t_s} = 0,
    \end{cases}$$
    Notice that the values of $\alpha$ and the index $i$ are not independent. It only has a solution if $i$ satisfies
    \begin{equation}\label{sistema dois pesos}
        \begin{cases}
        \gamma^{i t_1} = \dots = \gamma^{i t_s},\\
        \gamma^{i t_1} \in \F_p^*,
    \end{cases}
    \end{equation}
    in which case, $\alpha = - \gamma^{-it_1}$.
    The first equation in (\ref{sistema dois pesos}) is equivalent to
    \begin{equation}\label{indices validos 1}
        i (t_k - t_1) \equiv 0 \pmod{q - 1}
    \end{equation}
    for $2 \le k \le s.$ For each $k$, (\ref{indices validos 1}) has at least one solution if and only if $\gcd( t_k - t_1, q - 1) > 1$, and the indices for which this is true are the multiples of $(q- 1)/ \gcd(t_k - t_1, q - 1)$.

    The second equation in (\ref{sistema dois pesos}) is satisfied if and only if $\gcd(t_1 (p - 1), q - 1) > 1$, which is always true, with the possible exception of when $p  = 2$.  The values of $i$ for which this condition is satisfied are the multiples of $(q - 1)/ \gcd(t_1 (p - 1), q - 1)$. Thus, all the conditions in (\ref{sistema dois pesos}) are satisfied if and only if
    $D(t_1,\dots, t_s) 
    = \gcd(q - 1, t_1 (p - 1), t_2 - t_1, t_3 - t_1, \dots, t_s - t_1) > 1$,
    and the indices $i$ that satisfy the system are the multiples of $(q - 1) / D(t_1, \dots, t_s)$ smaller than $q - 1$, and each one of those will give us a word of the form $c(i;\alpha)$. There are $D(t_1, \dots, t_s)  - 1 $ such multiples. We can then generate every word with weight $2$ in the code by doing $n = q - 1$ shifts and multiplying by the $p - 1$ elements in $\F_p^*$. This will produce each codeword twice, thus
    $$A_2 = \frac{(p - 1) (q - 1) ({D(t_1, \dots, t_s)}  - 1)}{2}.$$

\end{proof}

For instance, let $p = 2$, $q = 2^3$ and consider the binary cyclic code $C_{1, 5}$ of length $n = 7$ over $\F_2$. Since $D(1,5) = 1$, this code has no codewords with weight $2$.

For a case where the minimum distance is $2$, let $p = 3$, $q = 3^2$ and consider the cyclic code $C_{1, 5}$ of length $n = 8$ over $\F_3$. We have
$$D(1,5) = \gcd(9 - 1, 1(3 - 1), 5 - 1) =  2,$$
thus $A_2 = 8$. The codewords with weight $2$ in this code are $(1, 0, 0, 0, 1, 0, 0, 0)$, its shifts and their multiples.

\subsection{Codewords with weight 3}

We will now introduce a sufficient condition for a cyclic code of the form $C_{t_1, \dots, t_s}$ to have codewords with weight $3$. Note that it is not a necessary condition, so it is not an exhaustive characterization of all cyclic codes with a minimum distance less or equal to 3.

Let us denote by $K_g(t)$ the $p$-cyclotomic coset of $t \pmod{p^g - 1}$, i.e.,
$$K_g(t) = \{ t p^k \pmod{p^g - 1} : k = 0 , \dots, g - 1 \}.$$
We say an integer $0 \leq i \leq p^m - 2$ belongs to $K_g(t)$ if there is an integer $0 \leq j \leq g - 1$ such that $i p^j \equiv t \pmod{p^g - 1}.$


The following Theorem is a generalization for codes of arbitrary characteristic of the bound presented in \cite[Theorem 1, Theorem 3]{charpin} for binary codes.

\begin{theorem}\label{peso 3 K_g}
    Let $p$ be an odd prime. Let $t_1, \dots, t_s$ be integers that are not in the same coset $K_m(t)$ but are in the same coset $K_g(t)$, where $t \not\equiv 0 \pmod{p^g - 1}$ and $g$ is a fixed divisor of $m$. Then the cyclic code $C_{t_1, \dots, t_s}$ has a minimum distance $d \leq 3$. 

   Moreover, the number $A_3$ of codewords with weight $3$ for $C_{t_1, \dots, t_s}$ satisfies the following bound:
    $$A_3 \ge \frac{(p-1) (p^m - 1)}{6} \left[ (p - 1)^2 (p^g - 2) - (p-2) (3p - 5)\right].$$
\end{theorem}
\begin{proof}
    Let $\beta = \gamma^u$, where $u = (p^m - 1)/(p^g - 1)$. Consider $a,b \in \F_p^*,\, 1 \le i, j \le p^g - 2$, where $i \ne j$, such that
    \begin{equation}\label{condicao peso 3}
        1 + a\beta^i + b\beta^j = 0.
    \end{equation}
    
    For each such choice of values we can define a weight 3 codeword of the form
    \begin{equation}\label{polinomio construido}
        c(x) = 1 + ax^{u i t^{-1}} + bx^{u j t^{-1}},
    \end{equation}
    where the inverse $t^{-1}$ is calculated in the ring $\Z_{p^g - 1}$. To check if $c(x)$ is a codeword in $C_{t_1, \dots, t_s}$, according to Lemma \ref{criterio palavra em codigo ciclico} we need to verify whether $\gamma^{t_k}$ are roots of the polynomial $c(x)$ for $1 \le k \le s$.  Since $t_k \in K_g(t)$, there are integers $c_1$ and $c_2 \ge 0$ such that $t_k = c_1 (p^g - 1) + p^{c_2} t$. Thus,
    \begin{align*}
        c(\gamma^{t_k}) &= 1 + a\gamma^{u t_k i t^{-1}} + b\gamma^{u t_k j t^{-1}}\\
        &= 1 + a\beta^{t_k i t^{-1}} + b\beta^{t_k j  t^{-1}}\\
        &= 1 + a\beta^{p^{c_2} t i t^{-1}} + b\beta^{p^{c_2} t j t^{-1}}\\
        &= 1 + a\beta^{p^{c_2} i} + b\beta^{p^{c_2}j }\\
        &= (1 + a\beta^{i} + b\beta^{j})^{p^{c_2}}\\
        &= 0.
    \end{align*}
    Hence, $C_{t_1, \dots, t_s}$ has at least this codeword with weight $3$. Thus, counting the number of codewords of this form gives us a lower bound for the number of codewords with weight 3 in the code. Since $\beta$ generates $\F_{p^g}^*$, for each choice of $a,b \in \F_p^*$, $1 \le i \le p^g - 2$, if a corresponding $j$ such that the tuple $(a,b,i,j)$ satisfies the condition \eqref{condicao peso 3} exists, then it is unique and is computed as
    $$\beta^j = b^{-1}(-1 - a\beta^{i}).$$
    Thus, we need to count all triples $(a,b,i)$ such that a valid $j$ satisfying condition \eqref{condicao peso 3} exists, thereby generating a valid weight 3 codeword of the form \eqref{polinomio construido}. 

    We will consider the set
    $$S = \{(a,b,i): a, b \in \F_p^*, 1 \le i,j \le p^g - 1 \}$$
    of all possible tuples and then remove the values that are invalid. There are three possible cases that such a tuple must avoid to generate a valid weight 3 codeword of the form \eqref{polinomio construido}:
    \begin{enumerate}[(I)]
        \item $j$ does not exist;
        \item $j = 0$;
        \item $j = i$.
    \end{enumerate}
    We will determine the subsets of $S$ that satisfy each exception and then remove them:
    \begin{enumerate}[(I)]
    \item Since $j$ is defined uniquely by the equation
    $$\beta^j = b^{-1}(-1 - a\beta^{i}),$$
    it does not exist only if $b^{-1}(-1 - a\beta^{i}) = 0$, i.e., if $\beta^i = -a^{-1}.$ Since that implies $\beta^i \in \F_p^*$, then $i$ must be a multiple of $(p^g - 1)/(p - 1)$. There are $(p - 2)$ choices of $1\le i \le p^g - 2$ that satisfy this condition. Once $i$ is chosen, the value of $a$ is fixed, and the value of $b$ is free in $\F_p^*$, giving us $(p-1)$ choices of $b$. Thus, the set $S_1$ of exceptions of the form (I) has order $(p - 1) (p - 2)$.

    \item This case happens if $b^{-1} (-1 - a \beta^i) = \beta^0 = 1$, i.e., if $(1 + a \beta^i) = -b$. Similarly to the previous case, $i$ can be chosen as any multiple of $(p^g - 1)/(p - 1)$, yielding $(p - 2)$ choices for $i$. Then, $a$ can be chosen as any value in $\F_p^{*}$ different from $- \beta^{-i}$. The value of $b$ is then fixed by these choices. Thus, the size of the set $S_2$ that satisfies the exception (II) is $(p - 2)^2.$

    \item If we require $i = j$ for a solution of \eqref{condicao peso 3}, then we have $1 + (a + b) \beta^i = 0$, i.e., $(a + b) = -\beta^{-i}.$ Similarly to the previous cases, we have $(p -2)$ choices for $i$. Then, we cannot choose $a = -\beta^{-i}$ because that would imply $b = 0$. Hence, we have $(p - 2)$ choices for $a$, and those choices fix the value of $b$. Thus, the set $S_3$ of exceptions of the form (III) has order $(p - 2)^2$.
    \end{enumerate}
    Therefore, the number of tuples $(a,b,i,j)$ that satisfy condition \eqref{condicao peso 3} and generate valid weight 3 codewords is
    \begin{align*}
        |S| - |S_1| - |S_2| - |S_3| &= (p - 1)^2 (p^g - 2) - (p - 1) (p - 2) - 2(p-2)^2\\
        &= (p - 1)^2 (p^g - 2) - (p -2) (3p - 5).
    \end{align*}
    Notice that the tuples $(a,b,i,j)$ and $(b,a,j,i)$ generate the same codeword. Hence, the number of distinct codewords of the form \eqref{polinomio construido} is
    $$\frac{(p - 1)^2 (p^g - 2) - (p -2) (3p - 5)}{2}.$$
    Since these are codewords in a cyclic code, we can generate more codewords by multiplying by constants in $\F_p^*$, and performing up to $(p^m - 1)$ shifts. For each $d \in \F_p^*$ and $0 \le l \le p^m - 2$, a codeword of the form $d x^l c(x) \pmod{ x^{p^m - 1} - 1}$ is also a weight 3 codeword. When $d$ and $l$ run through their respective intervals for all codewords of the form \eqref{polinomio construido}, each codeword is generated three times. Thus, the number $A_3$ of codewords with weight $3$ is greater than or equal to 
    $$\frac{(p-1) (p^m - 1)}{6} \left[ (p - 1)^2 (p^g - 2) - (p-2) (3p - 5)\right].$$
\end{proof}

 We remark that this theorem does not apply to irreducible cyclic codes of the form $C_t$, because they do not satisfy the condition that the indices defining the generator polynomial are not in the same coset $K_m(t)$ but are in the same coset $K_g(t)$ for some $g \mid m$.

 For instance, let $p = 3$, $m = 2$, and consider the cyclic code $C_{1,5}$. The values $1,\,5$ are not in the same $p$-cyclotomic coset $\pmod{3^2 - 1}$, but they are both in $K_1(1)$. Thus, the result applies in this case for $g = 1$. Therefore
 $$A_3 \ge \frac{(3 - 1) (3^2 -1)}{6} \left[(3-1)^2 (3^1 - 2) - (3 - 2) (3^2 - 5) 
 \right] = \frac{(3 - 1) (3^2 -1)}{6} \left[ 4 - 4 \right] = 0.$$
 This is exactly as expected, since it can be verified that $A_3 =0$ for this code.
\newline \hfill



For another example, consider $p = 5$, $m = 2$, and $t_1, t_2 = 1,\, 9$. The values $1$ and $9$ are not in the same $5$-cyclotomic coset $\pmod{5^2 - 1}$, but are both $K_1(1)$. Thus,
$$A_3 \ge \frac{(5-1) (5^2 - 1)}{6} \left[ (5 - 1)^2 (5^1 - 2) - (5-2) (3\cdot 5 - 5)\right]= 288.$$
Remarkably, it can be verified that $A_3 = 288$ so the lower bound is attained.
We can repeat the procedure used in the proof to construct a codeword with weight $3$. We have $u = (25 - 1)/(5 - 1) = 6$ and can verify that $\beta = \gamma^6$ satisfies the equation
$$1 + 3 \beta + \beta^3 = 0,$$
thus $c(x) = 1 + 3x^6 + x^{18}$ is a codeword with weight $3$ in $C_{1,9}$.

\hfill\newline

We will now introduce a class of cyclic codes for which there is an interesting criterion for the existence of weight $3$ codewords.


Let us consider binary cyclic codes $C_{1,t}$ of length $n = q - 1 = 2^m - 1$ and generator polynomial $g(x) = g_1(x) g_t(x)$. From Theorem \ref{peso 2} we know that this code cannot have codewords with weight 2. So the next step in determining the existence of low-weight codewords is to determine if the minimum distance is $3$. The parity check matrix for this polynomial is
\begin{equation}\label{H para C1t}
    H = \begin{bmatrix} 1 & \gamma^{1} & \gamma^{2} & \cdots & \gamma^{(n-1)}\\
1 & \gamma^{t} & \gamma^{2 t} & \cdots & \gamma^{(n-1) t}\end{bmatrix}.
\end{equation}
Since the code is cyclic, a codeword can be shifted so the coordinates with ones are $0,i,j$ where $0 < i < j \le q- 2$. Multiplying this codeword by $H$, we have that the codeword is in the code if and only if
$$
\begin{cases}
    1 + \gamma^{i} + \gamma^{j} = 0,\\
    1 + \gamma^{it} + \gamma^{jt} = 0.
\end{cases}$$
Since the indices run from $1$ to $q - 2$, there are indices $i, j$ that construct a weight $3$ codeword if and only if there is a solution $(x, y)$ with $x, y \in \F_q\setminus\{ 0, 1\}$ to the system
$$\begin{cases}
    1 + x + y = 0,\\
    1 + x^t + y^t = 0.
\end{cases}$$
 We can isolate $y = x + 1$ and substitute it into the second equation to obtain that the condition is that the polynomial
$$U_t (x) := 1 + x^t + ( 1 + x)^t$$
should have a root in $\F_q \setminus \{0, 1 \}$. 

\begin{theorem}\label{teorema fatorar peso 3}
    Let $C_{1,t}$ be the binary cyclic code with minimal polynomial $g(x) =$ $ g_1(x) $ $g_t(x)$, and let us suppose that the irreducible factors of
    $$U_t(x) = 1 + x^t + (1 + x)^t$$
    are $x, (x+1), f_1(x), \dots, f_l(x)$. Then the code has a codeword with weight $3$ if and only if at least one of the irreducible factors $f_1, \dots, f_l$ has a root in $\F_q$.

    Specifically, if $m_j := \deg{f_j} (x)$ for $1 \le j \le l$, then
    \begin{enumerate}[i)]
        \item If $m_j \nmid m$ for all $1 \le j \le l$, then the code has no codewords with weight $3$.

        \item If $m_j \mid m$ for some $1 \le j \le l$, then the code has codewords with weight $3$.
    \end{enumerate}

\end{theorem}
\begin{proof}
    The roots of $U_t(x)$ in $\F_q \setminus\{ 0, 1\}$, if they exist, must come from the irreducible factors other than $x$ and $x + 1$.
    The roots of an irreducible polynomial in $\F_2[x]$ must be in the extension field of degree equal to the degree of the polynomial. Thus, if $m_j = \deg{f_j(x)}$, the roots are in $\F_q$ if and only if $\F_{2^{m_j}} \subseteq \F_{q}$, which occurs if and only if $m_j \mid m$. Therefore, if no irreducible polynomial has a degree dividing $m$, then there are no codewords with weight $3$ in the code $C_{1,t}$. Conversely, if $\deg{f_j}(x) \mid m$ for some $1 \le j \le l$ then the code has a codeword with weight $3$.
\end{proof}

For instance, no matter the integer $m>2$ such that $q = 2^m$, the code $C_{1,3}$ with length $n = q - 1$ must have a minimum distance greater than $3$, since $U_3(x) = x^2 + x$ only has $0$ and $1$ as roots.

\begin{corollary}\label{corolario fatorar peso 3}
    Let $C_{1,t}$ be the binary cyclic code with minimal polynomial $g(x) =$ $ g_1(x) $ $g_t(x)$. Then $C_{1,t}$ has a minimum distance equal to $3$ if and only if
    \begin{equation*}
        \gcd(U_t(x),\, x^q + x) \ne x (x + 1).
    \end{equation*}
\end{corollary}
\begin{proof}
    The factors $x, \, x+1$ are present in both $U_t(x)$ and $x^q + x$. Any irreducible factor in $U_t(x)$ has degree dividing $m$ if and only if its roots are in $\F_q$, which occurs if and only if it divides $x^q  + x$. It is straightforward to verify that $x^2 + x$ divides both $U_t(x)$ and $x^q + x$, and that the multiplicity on $U_t(x)$ is $1$. Thus, there are other irreducible factors if and only if $\gcd(U_t(x),\, x^q + x) \ne x (x + 1).$
\end{proof}

The factoring of polynomials over finite fields is often challenging, with few cases where it can be done without exhaustive search. The following corollary presents a case where the minimum distance greater than $3$ can be easily verified.


\begin{corollary}\label{caso facil sem peso 3}
    Let $m = p_1^{e_1} \cdots p_s^{e_s}$ be the decomposition of $m$ into prime factors. If $t < p_i+ 3$ for all $1 \le i \le s$, then the code $C_{1,t}$ has minimum distance $d > 3$.
\end{corollary}
\begin{proof}
    Let us suppose that the irreducible factors of
    $$U_t(x) = 1 + x^t + (1 + x)^t$$
    are $x, (x+1), f_1(x), \dots, f_l(x)$, and denote $m_j = \deg{f_j}$ for all $1 \le j \le l$. Theorem \ref{teorema fatorar peso 3} tells us that the code has a minimum distance of $3$ if and only if $m_j \mid m$ for at least one $1 \le j \le l$. This can only happen if at least one of the primes $p_i$ divides $m_j$. Since the degree of $U_t(x)/ (x (x+1))$ is $t - 3$, the condition that $p_i > t - 3$ for every $1 \le i \le s$ implies that $p_i$ is greater than $m_j$ for every $1 \le j \le l$, since $t - 3 \ge m_j$. This implies that no prime factor of $m$ divides any of the degrees of the irreducible factors, and thus no degree divides $m$. Therefore, there are no codewords with weight $3$.
\end{proof}
For instance, let $q = 2^{31}$ and $t = 33$. Since $t < 31 + 3 = 34$, Corollary \ref{caso facil sem peso 3} implies that there are no codewords with weight $3$ in the code.

We remark that Corollary \ref{caso facil sem peso 3} is a generalization of a similar result where $m$ is a prime number, which can be seen in \cite[Proposition 3]{charpin}.

Thus the number of codewords of weight $3$ in this polynomial is related to the number of roots of some irreducible polynomials in $\F_2[x]$. To count the number of codewords of weight $3$ from the number of roots in $\F_q\setminus \{0, 1\}$ of $U_t(x)$, observe that each root $\gamma^i$ corresponds to another root $\gamma^j$ such that the codeword
$$c(x; i, j) = 1 + x^i + x^j$$
is in the code. Notice that we can swap $\gamma^i$ and $\gamma^j$ and still obtain the same codeword, so the number of codewords of the form $c(x;i, j)$ in the code is half the number of roots of $U_t(x)/(x (x+1))$ in $\F_q$. By shifting these codewords $n = q - 1$ times, we will count every codeword three times. Thus, the total number of codewords of weight $3$ is given by the following corollary:



\begin{corollary}\label{caso facil peso 3}
    The number of codewords with weight $3$ in $C_{1,t}$ is
    $$A_3 = \frac{(q - 1)  \left(\deg(\gcd(U_t(x), x^q + x)) - 2\right)}{6}.$$

\end{corollary}
\begin{proof}
    According to Corollary \ref{corolario fatorar peso 3}, every root of $U_t(x)$ that is in $\F_q$ is in $\gcd(U_t(x), x^q + x)$. By the discussion above we know that the total number of codewords is $(q - 1)/6$ times the total number of roots excluding $0$ and $1$.
\end{proof}

For instance, let $q = 2^4$, $t = 7$, and consider the binary cyclic code $C_{1,7}$ over $\F_q$, with the primitive element $\gamma$ satisfying $\gamma^4 + x + 1 = 0$. Since
$\gcd(U_7(x), x^{16} + x) = x^4 + x$, the Corollary (\ref{caso facil peso 3}) implies the code $C_{1,7}$ of length $n = 15$ has
$$A_3 = \frac{(16 - 1) \cdot (4 - 2)}{6} = 5$$
codewords with weight $3$. In fact, the roots of $U_7(x)$ in $\F_q\setminus\{0, 1\}$ are $\gamma^5$ and $\gamma^{10}$, which correspond to the codeword
$$(1, 0, 0, 0, 0, 1, 0, 0, 0, 0, 1, 0, 0, 0, 0),$$
and the codewords obtained by shifting it, totaling $5$ such codewords.

\section{Main results}
For here on we will regard only binary cyclic codes. Let $q = 2^m$, $\gamma$ a primitive element of $\F_q$, $C_{t_1,\dots, t_s}$ be the code of length $n = q - 1$ generated by the polynomial $g(x) = g_{t_1}(x) g_{t_2}(x)\cdots g_{t_s}(x)$ as defined previously, and let $A_0, A_1, A_2, \dots$ be its weight distribution. We will show a way to represent $A_w$, where $w \ge 1$, as the number of solutions of a system with some restrictions and show how it relates to $N_w$, the number of solutions in $\F_q$ of the system
\begin{equation}\label{N_w definicao}
\begin{cases}
    x_1^{t_1} + \cdots+ x_w^{t_1} = 0,\\
    \vdots\\
    x_1^{t_s} + \cdots+ x_w^{t_s} = 0.\\
\end{cases}
\end{equation}

In polynomial notation, codewords with weight $w$ are of the form
$$c(i_1, \dots, i_w) := x^{i_1} + \cdots + x^{i_w},$$
where $0 \le i_1 < i_2 < \dots < i_w \le q - 2$ are ordered indices. Multiplying any of these codewords by the parity check matrix (\ref{criterio palavra em codigo ciclico}), we conclude that the codeword is in $C_{t_1, \dots, t_s}$ if and only if
$$
\begin{cases}
    \gamma^{i_1 t_1} + \cdots+ \gamma^{i_{w} t_1} = 0,\\
    \vdots\\
    \gamma^{i_1 t_s} + \cdots+ \gamma^{i_{w} t_s} = 0.\\
\end{cases}$$
For a fixed primitive element $\gamma \in \F_q$, we define the discrete logarithm function $\log_\gamma(x): \F_q \rightarrow \Z_{\ge 0}$ as the function that gives the smallest non-negative integer such that $\gamma^{\log_\gamma(x)} = x$. Since $\gamma$ is a primitive element and the indices run from $0$ to $q - 2$, each $\gamma^{i_j}$ corresponds to a $x_j \in \F_q^*$. Thus, $A_w$ is the number of solutions over $\F_q^*$ of the system (\ref{N_w definicao}), where the indices are ordered by their discrete logarithms, i.e, $\log_{\gamma} x_1 < \log_{\gamma} x_2 <  \dots < \log_{\gamma} x_w.$ We remark that these solutions have no repetition of variables, unlike $N_w$, where the solutions can include zeros, be unordered, and may have repetitions. 

We will introduce notation for orderings and repetitions within tuples. A tuple with elements in $\F_q$ is said to be ordered if its coordinates are arranged in ascending order according to the discrete logarithm with respect to $\gamma$, with $0 \in \F_q$ placed in the first position if present. Ordered tuples with integer elements are ordered in the usual ascending order. A subtuple of an ordered tuple is a tuple formed by selecting a subset of the elements from the original tuple while preserving their relative order. For instance, $(2, 4)$ is a subtuple of $(1,2,3,4)$.

We denote the sum of coordinates in a tuple $v$ as $|v|$, and the number of coordinates as $\#v$. We call a tuple $w$-sum if the sum of its coordinates is $w$. For instance, for $v = (2,2,4,4,4,8)$ we have $|v| = 24$ and $\#v = 6$.

We introduce a notion for repetition within tuples, which we will call the partition of the tuple. Let $v = (v_1, \dots, v_w)$ be a tuple, and let $k$ denote the number of distinct values appearing in its coordinates. The partition $P(v)$ is defined as the ordered integer tuple $P(v) = (a_1, \dots, a_{k}),$
where each $a_i \in \Z_{>0}$ represents the number of repetitions of one of the distinct values. For instance, the partition for the tuple $v = (4,3,5,4,3,0,4)$ is $P(v) = (1,1,2,3)$, since $0$ and $5$ appear once each, $3$ repeats twice and $4$ repeats three times.


We will also refer to ordered tuples with positive integer coordinates as partitions, since they can be obtained as partitions of other tuples. For a partition tuple $v$, we define $\varepsilon(v)$ and $\theta(v)$ as the subtuples of $v$ consisting only of the even and odd coordinates in $v$, respectively.

For a positive integer $w$, its assembly $S(w)$ is defined as the set of all ordered positive $w$-sum partitions, i.e.,
$$S(w) = \left\{(v_1, \dots, v_k):  1 \le v_1 \le \dots \le v_k \le n ; \sum_{i = 1}^k v_i = w \right\}.$$
For instance, $S(4) = \{ (1,1,1,1), (1,1,2), (1,3), (2, 2), (4)\}.$

For a tuple $v = (a_1, \dots, a_{k})$ with sum $|v|= w$, we define its factorial as $v! = a_1! \cdots a_k!$, and its multinomial as
$$\binom{w}{v} := \binom{w}{a_1, a_2, \dots, a_k} = \frac{w!}{a_1! a_2! \cdots a_k!}.$$

We can now state our main result. The following theorem relates the weight distribution of a cyclic code to the number of solutions of a system of diagonal equations.

\begin{theorem}\label{teo relacao N_w A_w}
    Consider the binary cyclic code $C_{t_1, \dots, t_s}$ of length $n = q - 1$ with generator polynomial $g(x) = g_{t_1}(x) \cdots g_{t_s}(x)$. Let $A_0, A_1, A_2, \dots$ be the weight distribution of the code and, for any positive integer $w$, let $N_w$ be defined as in (\ref{N_w definicao}). The following relationship between $N_w$ and the weight distribution of $C_{t_1, \dots, t_s}$ holds:
    \begin{equation}\label{eq teo relacao}
        N_w = \sum_{v \in S(w)}  (A_{\#\theta(v)} + A_{\#\theta(v) - 1})  \frac{\#\theta(v)!}{P(\theta(v))!} \cdot  \binom{q - \#\theta(v)}{\#\varepsilon(v)} \frac{\#\varepsilon(v)!}{P(\varepsilon(v))!}\cdot  \binom{w}{v}.
    \end{equation}
    
\end{theorem}
\begin{proof}

    Let $N_w$ be the set of solutions to (\ref{N_w definicao}) over $\F_q$, with no restrictions. We will first divide this set into the union of subsets $N_w(v)$, where each subset contains solutions that are tuples of length $w$ with a partition tuple equal to $v$, for each partition $v$ with sum $w$. We will show how to relate the values of $N_w(v)$ to the weight distributions $A_k$ where $k \le w$.

    Let us fix a particular partition $v \in S(w)$ and generate every solution that has this partition. We will determine what values can be put into the coordinates to form a solution, how many repetitions they have, and in how many ways they can be arranged.

    Let $\theta(v)$ and $\varepsilon(v)$ be the ordered subtuples containing the odd and even values in $v$, respectively. The values with even repetition vanish in (\ref{N_w definicao}), while the values with odd repetition reduce to only one summand in (\ref{N_w definicao}) and thus must be solutions to
    \begin{equation}\label{N_w impares}
    \begin{cases}
        y_1^{t_1} + \cdots + y_{\#\theta(v)}^{t_1} = 0,\\
        \vdots\\
        y_1^{t_s} + \cdots + y_{\#\theta(v)}^{t_s} = 0,\\
    \end{cases}
    \end{equation}
    with the added restriction that they must be distinct. There are then two possibilities: if one of the values is $0$, the remaining $\#\theta(v) - 1$ values correspond to codewords with weight $\#\theta(v) - 1$, giving us $A_{\#\theta(v) - 1}$ choices; if none of the values is $0$, then there are $A_{\#\theta(v)}$ choices for the values. Thus, the total number of ordered tuples of values that are solutions to (\ref{N_w impares}) is $(A_{\#\theta(v) - 1} + A_{\#\theta(v)})$. Next, we need to choose how many repetitions each value has. We can select one of the odd repetition values in $\#\theta(v)$ for each ordered tuple of values with odd repetition, which can be done in $\#\theta(v)!/ P(\theta(v))!$ ways.

    The values with even repetition vanish in (\ref{N_w definicao}) so for any choice of values with odd repetition the only restriction is that we cannot choose values that have already been chosen. Thus, we have 
    $$\binom{q - \#\theta(v)}{\#\varepsilon(v)}$$
    tuples of values with even repetition. Similarly to the values with odd repetition, the number of repetitions of the values with even repetition can be chosen $\#\varepsilon(v)!/ P(\varepsilon(v))!$ ways. 

    Finally, to form a solution tuple $(x_1, \dots, x_w)$ of (\ref{N_w definicao}), we just need to distribute these coordinates according to the partition, which can be done $\binom{w}{v}$ ways. 

    The total number of solutions with the partition $v$ is then obtained multiplying every independent choice, giving us
    $$N_w(v) = (A_{\#\theta(v)} + A_{\#\theta(v) - 1})  \frac{\#\theta(v)!}{P(\theta(v))!} \cdot  \binom{q - \#\theta(v) }{ \#\varepsilon(v)} \frac{\#\varepsilon(v)!}{P(\varepsilon(v))!}\cdot  \binom{w}{ v}.$$
    Summing it over every $w$-sum partition gives us the value expression (\ref{eq teo relacao}).
\end{proof}


In the next section we will see some applications of this result.

\section{Applications}
Theorem \ref{teo relacao N_w A_w} relates the number of solutions of systems of diagonal equations to the weight distribution of binary cyclic codes. In this final section, we apply it in both directions: first we use binary cyclic codes with known weight distributions to determine the number of solutions of systems of diagonal equations; and then, we use a family of systems of diagonal equations with known number of solutions to find the distribution of low-weight codewords in a family of binary cyclic codes.

Note that for any non-trivial cyclic code $C_{t_1, \dots, t_s}$ the values $A_0 = 1$ and $A_1 = 0$ are already known. Therefore, if we know the values for $N_w$ for every positive integer $w$, we can iterate from $w = 2$ to determine each $A_w$ using the values of $N_w$ and the values of $A_j$ for $j < w$. Likewise, if we have the weight distribution $A_0, A_1,\dots$ for a code, we can compute the value of any $N_w$. The most computationally demanding part is determining $P(w)$, which becomes more difficult as $w$ increases.



For instance, a BCH code of designed distance $\delta$ is a cyclic code with generator polynomial $g(x) = \mbox{lcm}(g_{b}(x), g_{b + 1}(x), \dots, g_{b + \delta - 2}(x))$ where $b$ is a positive integer. A classic result about such BCH codes is that their minimum distance $d$ is greater or equal to $\delta$ \cite[Theorem 5.1.1]{error correcting codes}. 

Let $b$ and $\delta$ be parameters for a binary BCH code, $0 < w < d$ an odd integer and $N_w$ be the number of solutions of the system
$$\begin{cases}
    x_1^{b} + \cdots + x_w^{b} = 0,\\
    x_1^{b + 1} + \cdots + x_w^{b + 1} = 0,\\
    \vdots\\
    x_1^{b + \delta - 2} + \cdots + x_w^{b + \delta - 2} = 0.\\
\end{cases}$$
Since $\#\theta(v) \le w$ for every partition and $A_w = 0$ for $0< w < d$, the term $(A_{\#\theta(v)}+ A_{\#\theta(v) - 1})$ will vanish whenever $\#\theta(v) \notin \{0, 1\}$, significantly reducing the number of terms in (\ref{eq teo relacao}). When $w$ is odd, there are no partitions of $w$ without odd repetitions, so every summand that does not vanish must have $\#\theta(v) = 1$, further reducing the number of terms. In this case, every summand will be of the form
$$\binom{q - 1 }{ \#\varepsilon(v)} \frac{\#\varepsilon(v)!}{P(\varepsilon(v))!} \cdot \binom{w }{ v }.$$

An example of a binary BCH code is the code with length $n = 15$, designed distance $\delta = 5$, and generator polynomial $g(x) = \mbox{lcm}(g_1(x), g_2(x), g_3(x), g_4(x)) = g_1(x) g_3(x)$. Using the general example in this instance we will determine the number $N_3$ of solutions to the system
$$\begin{cases}
    x_1 + x_2 + x_3 = 0,\\
    x_1^2 + x_2^2 + x_3^2 = 0,\\
    x_1^3 + x_2^3 + x_3^3 = 0,\\
    x_1^4 + x_2^4 + x_3^4 = 0.\\
\end{cases}$$
The partition $S(3) = \{(1,1,1), (1,2), (3)\}$ is such that the only tuples that have $\#\theta(v) = 1$ are $(1,2)$ and $(3)$, giving us only two terms we will need to sum.
\begin{align*}
    N_3 =&\, \binom{15 }{ 1} \frac{1!}{1!} \cdot \binom{3 }{ 1, 2} + \binom{15 }{ 0} \binom{3 }{ 3}\\
    =&\, 46.
\end{align*}

\hfill \newline

Let us use our main result the other way around, i.e., describe an example where the values for $N_w$ are known and use that to determine the weight distribution for low-weight codewords over $\F_q$. For this purpose we will cite the following theorem that gives us the number of solutions for a family of systems of diagonal equations and then compute the distribution of some weights in the corresponding families of cyclic codes.

\begin{theorem}\label{teorema do exemplo}
    Let $w > 1$, $k$ and $f$ be positive integers with $\gcd(k, 2f + 1) = 1$. Let $\F_q$ be the finite field with $q = 2^{2f + 1}$ elements. Then the number $N_w$ of solutions of the system
    \begin{equation}
    \begin{cases}
        x_1 + \cdots + x_w = 0,\\
        x_1^{2^k + 1} + \cdots + x_w^{2^k + 1} = 0,
    \end{cases}        
    \end{equation}
    over $\F_q$ is given by
    \begin{equation*}
        N_w = \begin{cases}
            q^{w - 2} + (q - 1) 2^{(w - 2)(f+1)}, &\text{ if } w \text{ is even},\\
            q^{w - 2} + (q - 1) 2^{(w - 3)(f + 1) + 1}, &\text{ if } w \text{ is odd}.
        \end{cases}
    \end{equation*}
\end{theorem}
\begin{proof}
    See \cite[Theorem 3.7]{artigo do exemplo}.
\end{proof}

We use this to prove the following result:

\begin{theorem}\label{principal cap final}
    Let $k$ and $f$ be positive integers with $\gcd(k, 2f + 1) = 1$. Let $q  = 2^{2f + 1}$ and $C_{1,2^k + 1}$ be the binary cyclic code of length $n = q - 1$. Then
    \begin{align*}
        A_3 &= 0,\\
        A_4 &= 0,\\
        A_5&=\frac{(q-1)(q-2)(q-8)}{120},\\
        A_6&=\frac{(q - 1) (q - 2) (q - 6) (q - 8)}{720}.
    \end{align*}
\end{theorem}
\begin{proof}
    As discussed before, any binary cyclic code that is not the entire vector space has $A_0 = 1$ and $A_1 = 0$. Theorem \ref{peso 2} implies that $A_2 = 0$, and by our convention $A_{-1} = 0$. Thus, we will only compute the frequencies of weights equal to or greater than $3$.

    Using Theorem \ref{teorema do exemplo}, we have that
    \begin{equation}\label{eq N_w especifico}
        \begin{split}
            N_3 &= 3q - 2,\\
            N_4 &= q^2 + 2^{2(f+1)}(q - 1),\\
            N_5 &= q^3 + 2^{2(f+1) + 1}(q - 1).\\
            N_6 &= q^{4} + (q - 1) (2^{2(f+1)})^2
        \end{split}
    \end{equation}
    Notice that since $q = 2^{2f+1}$, then $2^{2(f+1)} = 2q$, thus
    \begin{equation}
        \begin{split}
            N_3 &= 3q - 2,\\
            N_4 &= 3q^2 - 2q,\\
            N_5 &= q^3 + 4q^2 - 4q.\\
            N_6 &= q^{4} + 4q^3 - 4 q^{2} 
        \end{split}
    \end{equation}
    
    We will use Theorem \ref{teo relacao N_w A_w} for the recursive relation. For $w = 3$, we have
$$S(3) = \{ (1,1,1), (1,2), (3)\},$$
and thus,
\begin{align*}
    N_3 =& \,(A_3 + A_2) \frac{3!}{3!} \binom{q - 3}{0} \frac{0!}{0!}  \binom{3}{1, 1, 1}  \\&+  (A_1 + A_0) \frac{1!}{1!} \binom{q - 1}{1} \frac{1!}{1!} \binom{3}{1, 2}  \\&+  (A_1 + A_0) \frac{1!}{1!} \binom{q}{0} \frac{0!}{0!} \binom{3}{3}\\
    =& \, 3! A_3 + 3 (q - 1) + 1.
\end{align*}
Consequently, 
$$A_3 = \frac{N_3 - 3(q- 1) - 1}{3!}.$$
Thus, the value of $N_3$ is tied to the value of $A_3$.
Using \eqref{eq N_w especifico}, we obtain $A_3 = 0.$

For $w = 4$, we have
$$S(4) = \{ (1,1,1,1), (1,1,2), (1,3), (2,2), (4)\}.$$
Hence,
\begin{align*}
    N_4 = & \, (A_4 + A_3) \frac{4!}{4!} \binom{4}{1,1,1,1}\\
    &+ (A_2 + A_1) \frac{2!}{2!} \binom{q - 2}{1} \frac{1!}{1!} \binom{4}{1,1,2}\\
    &+ (A_2 + A_1) \frac{2!}{1!1!} \binom{4}{1,3}\\
    &+ \binom{q}{2} \frac{2!}{2!} \binom{4}{2, 2}\\
    &+ \binom{q}{1} \frac{1!}{1!} \binom{4}{4}\\
    =&\, 4!(A_4 + A_3) +  3 q (q-1) + q,
\end{align*}
and thus
$$A_4 = \frac{N_4 -  3 q (q-1) - q}{4!} - A_3.$$
Using \eqref{eq N_w especifico}, we get
$$A_4 = 0.$$

For $w = 5$, we have
$$S(5) = \{ (1,1,1,1,1), (1,1,1,2), (1,1,3), (1,4) (1,2,2), (2,3), (5)\},$$
hence,
\begin{align*}
    N_5 = & \, (A_5 + A_4) \frac{5!}{5!} \binom{5}{1,1,1,1,1}+ (A_3 + A_2) \frac{3!}{3!} \binom{q - 3}{1} \frac{1!}{1!} \binom{5}{1,1,1,2}\\
    &+ (A_3 + A_2) \frac{3!}{1!2!} \binom{5}{1,1,3} + (A_1 + A_0) \frac{1!}{1!}\binom{q - 1}{1} \frac{1!}{1!} \binom{5}{1, 4}\\
    &+ (A_1 + A_0) \frac{1!}{1!}\binom{q - 1}{2} \frac{2!}{2!} \binom{5}{1,2, 2}+ (A_1 + A_0) \frac{1!}{1!} \binom{q - 1}{ 1} \frac{1!}{1!}\binom{5}{2,3}\\
    &+ (A_1 + A_0) \frac{1!}{1!}\binom{5}{5}\\
    =&\, 5!(A_5 + A_4) + 60 (q-3) A_3 + 60 A_3 + 5(q-1) + 15(q-1)(q-2) + 10(q-1) + 1\\
    =&\, 5!(A_5 + A_4) + 60(q - 2)A_3 + 15 (q - 1)^2 + 1.
\end{align*}
Thus,
$$A_5 = \frac{N_5 - 60(q - 2)A_3 - 15 (q - 1)^2 - 1 }{120} - A_4,$$
and \eqref{eq N_w especifico} implies that
$$A_5 = \frac{(q - 1) (q - 2) (q - 8)}{120}.$$

For $w = 6$, we have
{\small $$S(6) = \{ (1,1,1,1,1,1), (1,1,1,1,2), (1,1,1,3), (1,1,4) (1,1,2,2), (1,2,3), (1,5), (2,2,2), (2,4), (3,3), (6)\},$$}
hence,
\begin{align*}
    N_6 = &\,(A_6 + A_5) \frac{6!}{6!} \binom{6}{1,1,1,1,1,1}+ (A_4 + A_3) \frac{4!}{4!} \binom{q - 4}{1} \frac{1!}{1!} \binom{6}{1,1,1,1,2}\\
    &+ (A_4 + A_3) \frac{4!}{1!3!} \binom{6}{1,1,1,3} + (A_2 + A_1) \frac{2!}{2!}\binom{q - 2}{1} \frac{1!}{1!} \binom{6}{1,1, 4}\\
    &+ (A_2 + A_1) \frac{2!}{2!}\binom{q - 2}{2} \frac{2!}{2!} \binom{6}{1,1,2, 2}+ (A_2 + A_1) \frac{1!}{1!} \binom{q - 2}{ 1} \frac{2!}{1!1!}\binom{6}{1,2,3}\\
    &+ (A_2 + A_1) \frac{2!}{1!1!}\binom{6}{1,5} + \binom{q}{3} \frac{3!}{3!}\binom{6}{2,2,2}\\
    &+  \binom{q}{2}\frac{2!}{1!1!}\binom{6}{2,4} +  (A_{2} + A_1) \frac{2!}{2!}\binom{6}{3,3} +  \binom{q}{1}\frac{1!}{1!}\binom{6}{6}\\
    =& 6!(A_6 + A_5)  + 90\binom{q}{3} + 30 \binom{q}{2} + q
\end{align*}
Thus,
$$A_6 = \frac{N_6 - 15q (q-1)(q-2) - 15 q(q-1) - q}{720} - A_5,$$
and \eqref{eq N_w especifico} implies that
\begin{align*}
    A_6 &= \frac{q^{4} + 4q^3 - 4 q^{2}  - 15q (q-1)(q-2) - 15 q(q-1) - q}{720} - \frac{(q - 1) (q - 2) (q - 8)}{120}\\
&=\frac{(q - 1) (q - 2) (q - 6) (q - 8)}{720}
\end{align*}
\end{proof}

Notably, this proves that the minimum distance for codes in this family is equal to or greater than $5$. For a concrete example, let $f = 2$, $k 
= 2$, and $q = 32$. The corresponding code is the cyclic code $C_{1,2^2 + 1}$ with length $n = 31$, which according to Theorem \ref{principal cap final} has a weight distribution that starts like this:
\begin{equation*}
        A_0 = 1, \quad
        A_1 = 0, \quad
        A_2 = 0, \quad
        A_3 = 0, \quad
        A_4 = 0, \quad
        A_5 = 186, \quad A_6 = 806.
\end{equation*}




\section*{Acknowledgements}

The first author was supported in part by the Coordena\c{c}\~ao de Aperfei\c{c}oamento de Pessoal de N\'ivel Superior-Brazil (CAPES) - Finance Code 001. The second author was partially supported by CNPq Grant 316843/2023-7 Brazil.

\end{document}